\documentclass[conference]{IEEEtran}
\IEEEoverridecommandlockouts

\usepackage{cite}
\usepackage{amsmath,amssymb,amsfonts}
\usepackage{algorithmic}
\usepackage{graphicx}
\usepackage{textcomp}
\usepackage{xcolor}
\def\BibTeX{{\rm B\kern-.05em{\sc i\kern-.025em b}\kern-.08em
    T\kern-.1667em\lower.7ex\hbox{E}\kern-.125emX}}
 
\newtheorem{theorem}{\textbf{Theorem}}

\newtheorem{proposition}[theorem]{\textbf{Proposition}}
\newtheorem{corollary}[theorem]{\textbf{Corollary}}

\newtheorem{lemma}[theorem]{\textbf{Lemma}}
\newtheorem{remark}[theorem]{\textbf{Remark}}

\newcommand{\R}{\mathbb{R}}

\newcommand{\N}{\mathbb{N}}
\newcommand{\C}{\mathbb{C}}

\newcommand{\D}{\mathbb{D}}

\renewcommand{\H}{\mathcal{H}}

\definecolor{darkviolet}{rgb}{0.58,0,0.83}

 \usepackage{comment}
 \allowdisplaybreaks
\begin{document}
\title{How large are the gaps in 
phase space?}

\author{
\IEEEauthorblockN{ Michael Speckbacher}
\IEEEauthorblockA{\textit{Acoustics Research Institute},
\textit{Austrian Academy of Sciences}\\
Dominikanerbastei 16, 1010 Vienna, Austria\\
Email: michael.speckbacher@oeaw.ac.at}
}
\date{}
\maketitle

\begin{abstract}
Given a sampling measure  for the wavelet transform (resp. the short-time Fourier transform) with the wavelet (resp. window) being chosen from the family of Laguerre (resp. Hermite) functions, we provide quantitative upper bounds on the radius of any  ball that does not intersect the support of the  measure. The estimates depend on the condition number, i.e., the ratio of the sampling constants, but are independent of the structure of the measure. Our proofs are completely elementary and rely on explicit formulas for the respective transforms.
\end{abstract}


\section{Introduction}

This article draws inspiration from the paper \emph{'How large are the spectral gaps?'} by Iosevich and Pedersen \cite{gaps}, as well as the recent work by  Papageorgiu and van Velthoven \cite{jordy24}, who gave quantitative estimates for relative denseness
for exponential frames and coherent frames on groups of polynomial growth, respectively. The main objective of this paper is to obtain  a similar result for the wavelet transform using a specific family of orthogonal functions as wavelets.   In particular, we 
establish an upper bound on the  radius of  any ball in wavelet phase space that does not intersect the support  of a sampling measure. {Our proof follows a different approach than  }\cite{jordy24} since the $(ax+b)$-group underlying the wavelet transform exhibits exponential growth. In addition, we discuss sampling measures for the short-time Fourier transform (STFT) and provide a bound  which, in principle, could readily be obtained using the ideas of \cite{jordy24}. Given that the proof strategy closely mirrors that of the  wavelet transform and that our elementary approach provides explicit constants, we chose to include the statement in this article.

Let $(X,\nu)$ be a metric measure space and $\H$ be a closed subspace of $L^2(X,\nu)$. A measure $\mu$ on $X$ is called a \emph{sampling measure} if there are constants $A,B>0$ s.t. 
$$
A\|F\|^2_{L^2_\nu}\leq \int_{X}| F(z)|^2d\mu(z)\leq B\|F\|^2_{L^2_\nu},\quad F\in \H.
$$
As we already pointed out, we  focus   on the cases when $\H$ is  the range of the wavelet transform or of the STFT.
A measure $\mu$ is called $(\gamma,R)$\emph{-dense} if   $\mu(B_R(z))\geq\gamma>0$ for every $z\in X$, and
\emph{relatively dense} if  there exist $R>0$ and $\gamma>0$ s.t. $\mu$ is $(\gamma,R)$-dense.    

There are various qualitative and quantitative results relating  sampling measures and relative denseness. 
For one, it is well-known   that  the support of a sampling measure for the wavelet transform as well as the STFT  is    necessarily relatively dense, see, e.g., \cite{Ascensi,japero87,lue98,orce98}. If a sampling measure for the STFT is discrete, i.e., $\mu=\sum_{\lambda\in\Lambda}\delta_\lambda$  with $\Lambda\subset \R^{2d}$ discrete, then  the density theorem for Gabor frames provides a more refined  analysis  assuring that the lower Beurling density of $\Lambda$ is greater or equal than one for every window function \cite{heil-density}.
For the wavelet transform, however, no such general result is known and all existing density  statements either require a particular structure of the sampling points and\,/\,or choice of the wavelet function  \cite{advances,jl&jordy,seip-regular,seip-disk}.  See \cite{kuty} for a discussion of the difficulties that occur in this setting.

{Previous 
quantitative results on relative denseness primarily focused on   estimates of the sampling constants $A$ and $B$ in terms of $R$ and $\gamma$. This approach was studied, e.g., in \cite{abspe18-sieve,wavelet-large-sieve,dom-sets,jaspe21}  for 
$\mu=\chi_\Omega \nu$ (here  $\chi_\Omega$ denotes the characteristic function of a measurable set $\Omega$),
or  in \cite{suzho02} for    a discrete measure   $\mu$ leading to explicit frame bounds. Here, we study the converse problem of providing an upper bound on $R$ in terms of $A$ and $B$ given that $\mu$ is not $(\gamma,R)$-dense.
}

We prove our main estimates in  Theorem~\ref{thm:gaps-wavelets} for a family of  orthogonal wavelets $\psi_n^\alpha$, $n\in\N_0$, $\alpha>0,$   defined in Fourier domain in terms of the generalized Laguerre polynomials (see equation \eqref{def-psi_n}). The ranges  of the respective wavelet transforms  play   important roles in various fields as they may be identified with     Bergman spaces of analytic functions on the upper half-plane and the unit disk ($n=0$), and with  hyperbolic Landau level spaces in quantum mechanics, see \cite{abbagomo15,abdoe12,wavelet-large-sieve} for    detailed discussions of these connections.
It seems that the statement of Theorem~\ref{thm:gaps-wavelets} was previously unknown, even in the literature on Bergman spaces. The only result, that we are aware of, that  exhibits  a   conceptual similarity, is   \cite[Lemma~4.1]{sei91}.

Our  proofs  follow the general strategy developed in \cite{gaps} and  rely on explicit formulas of the wavelet transform and the STFT.   This approach allowed  us to derive elementary  bounds on the quotient of, e.g., two wavelet transforms \begin{center}$z\mapsto  {W_{\psi_n^\alpha}\psi_0^\alpha(z)}/{ W_{\psi_n^\alpha}\psi_0^\alpha(w^{-1}\cdot z)}, \quad z,w\in\C^+,\ z\neq w,$\end{center}
on certain regions in   phase space which is then partitioned  accordingly.

\section{Sampling Measures for the Wavelet Transform}\label{sec:wavelets}

\subsection{Basic Wavelet Theory}
We   use the standard notation $\H^{2}(\mathbb{C}^{+})$ for  the \emph{Hardy space } of analytic functions in $\mathbb{C}^{+}$ equipped with the norm 
\begin{equation*}
\left\Vert f\right\Vert _{\H^{2}}^2:=\text{ }\sup_{0<s<\infty }\int_{-\infty }^{\infty }\left\vert
f(x+is)\right\vert ^{2}dx<\infty \text{.}
\end{equation*}
Let $z=x+is\in  \mathbb{C}^{+}$.   The \emph{time-scale shift} $\pi(z)$ of a function $\psi \in H^{2}(\mathbb{C}^{+}) $ is defined  as
\begin{equation*}
\pi (z)\psi (t):=T_{x}D_{s}\psi (t)=s^{-\frac{1}{2}}\psi (s^{-1}(t-x)).\label{representation}
\end{equation*}
One may  identify $\C^+$ with the $(ax+b)$-group via the     multiplication
\begin{equation}\label{eq:group}
z\cdot w= x+sx^{\prime }+iss^{\prime },\quad z=x+is,\ w=x^{\prime }+is^{\prime }\in \C^+\text{,}
\end{equation}
Then the neutral element is  $i$, and the inverse
element of $z$ is $z^{-1}=- {x}/{s}+i/{s}$.  Throughout this paper "$\cdot$" will exclusively be used to  denote  the group multiplication \eqref{eq:group}. 
The \emph{wavelet transform} of $f\in \H^2(\C^+)$ with
respect to a wavelet\ $\psi $ is defined as 
\begin{equation*}
W_{\psi }f(z)=\left\langle f,\pi ({z})\psi \right\rangle . 
\end{equation*}
A wavelet $\psi$ is called \emph{admissible} if 
\begin{equation}
0<\int_{\R^+}\big| \widehat{\psi}(\xi)\big|^2\frac{d\xi}{\xi}=:C_{\psi }<\infty, \label{Adm_const}
\end{equation}
where $\widehat{\psi}$ denotes the Fourier transform of $\psi$.
For an admissible wavelet $\psi$, the wavelet transform $W_\psi:\H^2(\C^+)\to L^2(\C^+,s^{-2}dz)$ is a constant multiple of an isometry, i.e.,  
\begin{equation}
\int_{\mathbb{C}^{+}}\left\vert W_{\psi }f(z)\right\vert ^{2}s^{-2}dz=C_{\psi }\left\Vert f\right\Vert _{\H^{2} 
}^{2},  \label{isometry}
\end{equation}
where $dz$ denotes the Lebesgue measure on $\C^+$.
A family of vectors $\{\pi(\lambda)\psi\}_{\lambda\in\Lambda}\subset \H^2(\C^+)$   is called a \emph{wavelet frame} if there exist constants $A,B>0$ s.t. for every $f\in\H^2(\C^+)$
$$
A\|f\|^2_{\H^2}\leq \sum_{\lambda\in\Lambda}| W_\psi f(\lambda)|^2\leq B\|f\|^2_{\H^2}.
$$
A measure $\mu$ is called a \emph{sampling measure} for the wavelet transform $W_\psi$ if  for every $f\in\H^2(\C^+)$
$$
A\|f\|^2_{\H^2} \leq \int_{\C^+}| W_\psi f(z)|^2  d\mu(z)\leq B\|f\|^2_{\H^2}.
$$
In this terminology, a wavelet frame $\{\pi(\lambda)\psi\}_{\lambda\in\Lambda}$ corresponds to the sampling measure $\mu=\sum_{\lambda\in\Lambda}\delta_\lambda$, where $\delta_\lambda$ denotes the Kronecker delta.

\subsection{Pseudohyperbolic Metric and M{\"o}bius Transform}

\noindent The \emph{pseudohyperbolic metric} on $\mathbb{C}^{+}$ is given by 
\begin{equation*}
\rho_{\C^+} (z,w):=\left\vert \frac{z-w}{z-\overline{w}}\right\vert ,\qquad
z,w\in \mathbb{C}^{+},
\end{equation*}%
and the \emph{pseudohyperbolic disk} of radius $R>0$ centered at $z\in 
\mathbb{C}^{+}$ is denoted by $\mathcal{D}_{R}(z):=\{\omega \in \mathbb{C}%
^{+}:\ \rho_{\C^+} (z,w)<R\}$.
Note that $\rho_{\C^+} $ only takes values in the half open interval $
[0,1)$ and that $\rho_{\C^+} (z,w)=\rho_{\C^+} (z^{-1}\cdot w,i)$. 

Let $\mathbb{D}_{R}(z)\subset \C$ denote the Euclidean disk of radius $R>0$ centered at $z\in\C$. We
write for short $\mathbb{D}_R:=\mathbb{D}_R(0)$ and $\mathbb{D}:=\mathbb{D}_{1}$. The
Moebius transform $T:\mathbb{D}\rightarrow \mathbb{C}^{+}$,  
\begin{equation*}
T(u):=i\frac{1+u}{1-u},\quad u\in \mathbb{D},
\end{equation*}%
is bijective and maps the pseudohyperbolic distance in $\mathbb{C}^{+}$ to the
pseudohyperbolic distance in $\mathbb{D}$, i.e., for $u,v\in\mathbb{D}$
\begin{equation*}
\rho_{\C^+} (T(u),T(v))=\left\vert \frac{v-u}{1-\overline{u}v}\right\vert
=:\rho _{\mathbb{D}}(u,v).
\end{equation*}
 
\subsection{Quantitative Bounds for Gaps in $\C^+\backslash \emph{supp}(\mu)$}
 
  We will establish our main  result for wavelets chosen from the orthonormal basis $\{\psi
_{n}^{\alpha }\}_{n\in\N_0}\subset \H^{2}(\mathbb{C}^{+})$, $ {\alpha>0,}$ which is
defined in the Fourier domain via
\begin{equation}\label{def-psi_n}
\widehat{\psi _{n}^{\alpha }}(t):={\sqrt{ \frac{ 2^{\alpha +2}\hspace{1pt}  \pi\hspace{1pt} n!}{\Gamma
(n+\alpha +1)}}}\ t^{\frac{\alpha }{2}}e^{-t}L_{n}^{\alpha }(2t),\quad t>0,
\end{equation}%
where $L_{n}^{\alpha }$ denotes the generalized Laguerre polynomial of
degree $n$, see, e.g., \cite[Chapter~18]{NIST10}
\begin{equation*}
L_{n}^{\alpha }(t)=\sum_{k=0}^{n}\frac{(-1)^{k}}{k!}\binom{n+\alpha }{n-k}%
t^{k},\quad t>0.
\end{equation*}%
 We need an explicit formula for the inner products of time-scale shifted versions of $\psi_0^\alpha$ and $\psi_n^\alpha$ which can be found, e.g., in  \cite[Proposition~1]{wavelet-large-sieve}.

\medskip
\begin{proposition}
\label{lem:explicit-wavelet-trafo} 
\textit{Let $\alpha>0$, and $z=x+is$ as well as $w=x^{\prime
}+is^{\prime }$ be in $\mathbb{C}^{+}$. For every $n\in \mathbb{N}_{0}$,
one has  
\begin{align*}
\langle \pi (w)\psi_n^\alpha,&\pi(z)\psi_0^\alpha\rangle   =c_n^\alpha  \left( \frac{z-w }{z-
\overline{w}}\right) ^{n}  \left( \frac{{2}\sqrt{ss^{\prime }}}{i(\overline{w}-z)}
\right) ^{\alpha +1},\notag
\end{align*}
where $c_n^\alpha=\sqrt{\Gamma(n+\alpha+1)/\Gamma(\alpha+1)n!\, }$ .
}
\end{proposition}
\medskip 
Upon applying the M\"obius transform, a straightforward computation shows that for  $u,v\in\mathbb{D}$  
\begin{align*}
\big|\langle \pi(T(v))&{\psi _{0}^{\alpha }},\pi(T(u))\psi _{n}^{\alpha } \rangle \big| \\ &=c_n^\alpha\, \rho_\D(u,v)^{n}\left(\frac{{(1-|u|^2)\, (1-|v|^2)}}{|1-\overline{u}v|^2}\right) ^{\frac{\alpha +1}{2}}.
\end{align*}
In  particular,   $\langle \pi(T(v))\psi _{0}^{\alpha },\pi(T(u))\psi _{n}^{\alpha }\big\rangle$ is nonzero whenever $u\neq v$ which implies that the following auxiliary function is well-defined for $u\neq v$
\begin{align*} 
H_n^\alpha(u,v):&= \left|\frac{ \big\langle \pi(T(0))\psi _{0}^{\alpha },\pi(T(u))\psi _{n}^{\alpha }\big\rangle }{ \big\langle \pi(T(v))\psi _{0}^{\alpha },\pi(T(u))\psi _{n}^{\alpha }\big\rangle }\right| ^{2}
\\ &=\left( \frac{\rho_\D(0,u)
 }{\rho_\D(u,v)}\right)^{2n} \left(\frac{|1-\overline{u}v|^2}{{ 1-|v|^2}} \right)^{\alpha+1} 
\\ &=  \frac{|u|^{2n} |1-\overline{u}v|^{2(n+\alpha+1)}}{|u-v|^{2n}( 1-|v|^2)^{\alpha+1}}.
\end{align*}
 Note that this expression is rotationally invariant, i.e., $H_n^\alpha(ue^{i\varphi},ve^{i\varphi})=H_n^\alpha(u,v)$ for any $\varphi\in(0,2\pi]$.
\medskip

\begin{lemma}\label{lem:aux}
 \textit{   Let $v=|v|$ and $u=|u|e^{-i\varphi}$ with  $0<|v|<R\leq |u|<1$. If $|\varphi|\leq (1-|v|R)$, 
    then 
    $$
    H_n^\alpha(u,v)\leq 2^{n+\alpha+1}\frac{(1-|v|R)^{2n+\alpha+1}}{(1-|v|/R)^{2n}}.
    $$
    }
\end{lemma}
\noindent\begin{proof}
First, we  note that $|1-\beta e^{i\varphi}|^2=1+\beta^2-2\beta\cos\varphi$. Therefore,
\begin{align}\label{eq:H_na-est}
H_n^\alpha(u,v)&=    \frac{|u|^{2n}\, \big|1-|uv|e^{i\varphi}\big|^{2(n+\alpha+1)}}{\big||u|e^{-i\varphi}-|v|\big|^{2n}\, ( 1-|v|^2)^{\alpha+1}} \notag 
\\
&\leq \frac{\big|1-|uv| e^{i\varphi}\big|^{2(n+\alpha+1)}}{(1-|v|/|u|)^{2n}\, (1-|v|R)^{\alpha+1}} \notag
\\ 
&\leq  \frac{(1+|uv|^2-2|uv|\cos\varphi)^{n+\alpha+1}}{(1-|v|/R)^{2n}\, (1-|v|R)^{\alpha+1}}. 
\end{align}
Observe that $2(1-\cos\varphi) \leq \varphi^2$ for every $\varphi\in\R$.  Using the assumption on $\varphi$ we hence deduce that
\begin{align*}
    1+|uv|^2-2|uv|\cos\varphi&=(1-|uv|)^2+2|uv|(1-\cos\varphi)\\ &\leq  (1-|v|R)^2+  \varphi^2
    \\
    &\leq 2(1-|v|R)^2.
\end{align*}
Plugging this estimate into \eqref{eq:H_na-est} shows
\begin{align*}
H_n^\alpha(u,v)&\leq   \frac{2^{n+\alpha+1}(1-|v|R)^{2(n+\alpha+1)}}{(1-|v|/R)^{2n}\, (1-|v|R)^{\alpha+1}}
\\ &=\frac{2^{n+\alpha+1}(1-|v|R)^{2n+\alpha+1}}{(1-|v|/R)^{2n} },
\end{align*}
which concludes the proof. \hfill $\Box$
\end{proof}
\medskip

With this auxiliary result in place, we are now ready to prove our main result. \medskip
\begin{theorem}\label{thm:gaps-wavelets}
\textit{Let $n\in\N_0$, $\alpha>0$, and $\mu$ be a sampling measure for the wavelet transform with wavelet 
$ \psi_n^\alpha$ and sampling constants $A,B>0$.}
 
\textit{ 
 If there exists $z\in\C^+$ s.t. $\mu( \mathcal{D}_R(z))=0$, then 
         \begin{equation}\label{eq:th-2}
    R\leq 1- \left(\frac{C_{n,\alpha}}{\pi}\ \frac{A}{ B}\right)^{1/\alpha},
  \end{equation}
  where }
  $$
  C_{n,\alpha}=\left\{\begin{array}{rl}
    4^{-(\alpha+1)},   & n=0,  \\
     6^{-(2n+\alpha+1)},  &  n\in\N.
  \end{array}\right. $$
\end{theorem}
\noindent\begin{proof} 
We partition $\mathbb{D}\backslash\{0\}$ into $K$ circular sectors $$S_k\hspace{-.5pt}=\hspace{-.5pt}\Big\{re^{i\varphi}\in\mathbb{D}\backslash\{0\}:\ \frac{\pi(2k-1)}{K}\leq\varphi<\frac{\pi (2k+1)}{K}\Big\},$$
$k=0,...,K-1,$ and pick $K$ points $v_k=re^{2\pi ik/K}\in S_k$ with $r<R$.  The appropriate choice  of $K$ and $r$ in terms of $R$ will be determined later.   Moreover, we set $P_k=T(S_k)$
  and point out that by a straightforward computation    $w\in (z\cdot P_k)\backslash \mathcal{D}_R(z)$ if and only if $T^{-1}(z^{-1}\cdot w)\in S_k\backslash \mathbb{D}_R$.  Therefore, if $\mu$ is a sampling measure s.t.  $\mu( \mathcal{D}_R(z))=0$, then
    \begin{align*}
        A&=A \|\pi(z)\psi_0^\alpha\|_{\H^2}^2 
        \\
        &\leq 
        \int_{\C^+}\big|\langle \pi(z)\psi_0^\alpha,\pi(w)\psi_n^\alpha\rangle |^2d\mu(w)
       \\
         &  =        \int_{\C^+ }\big|\langle \psi_0^\alpha,\pi(z^{-1}\cdot w)\psi_n^\alpha\rangle |^2 d\mu(w)
          \\
         &  =   \sum_{k=1}^K    \int_{z\cdot P_k}\left|\frac{ \langle  \psi_0^\alpha,\pi(z^{-1}\cdot w)\psi_n^\alpha\rangle  }{ \langle  \pi(T(v_k))\psi_0^\alpha,\pi(z^{-1}\cdot w)\psi_n^\alpha\rangle  }\right|^2 \times \\ & \hspace{.5cm}\times \big|\langle \pi(T(v_k))\psi_0^\alpha,\pi(z^{-1}\cdot w)\psi_n^\alpha\rangle \big|^2 d\mu(w)
       \\
       &= \sum_{k=1}^K    \int_{z\cdot P_k}H_n^\alpha\big(T^{-1}(z^{-1}\cdot w),v_k\big)\times \\ &\hspace{.5cm}\times  \big|\langle \pi(z\cdot T(v_k))\psi_0^\alpha,\pi( w)\psi_n^\alpha\rangle \big|^2 d\mu(w)
      \\
          &\leq  \sum_{k=1}^K    \sup_{u\in S_k\backslash \mathbb{D}_R}H_n^\alpha(u,v_k)\times \\ &\hspace{0.5cm} \times \int_{\C^+}\big|\langle \pi(z\cdot T(v_k))\psi_0^\alpha,\pi(w)\psi_n^\alpha\rangle \big|^2d\mu(w)
      \\
        &\leq  K  B    \sup_{u\in S_0\backslash \D_R }H_n^\alpha( u,v_0) ,
    \end{align*}  
    where we used that $\big\| \pi(z)  \psi_0^\alpha \big\|_{\H^2}=1$  and the rotational invariance of $H_n^\alpha$ in the final step.
     
Let us choose $r=R^\kappa$ for $\kappa>1$, and $K=\lceil\pi/(1-rR)\rceil$. Any $u=|u|e^{i\varphi}\in S_0\backslash \mathbb{D}_R$ satisfies $|u|\geq R$ as well as $|\varphi|\leq  \pi/K\leq (1-rR)$. We may thus apply Lemma~\ref{lem:aux} to show that
\begin{align*}
A&\leq\left\lceil\frac{\pi}{1-R^{1+\kappa}}\right\rceil\, {2^{n+\alpha+1}\, B}\, \frac{(1-R^{1+\kappa})^{2n+\alpha+1}}{(1-R^{\kappa-1})^{2n}} \\ &\leq \pi\,  {2^{n+\alpha+2}\, B}\,  \frac{(1-R^{1+\kappa})^{2n+\alpha}}{(1-R^{\kappa-1})^{2n}}.
\end{align*}
  This inequality simplifies to $A\leq   \pi\, {2^{\alpha+2}B}  {(1-R^{1+\kappa})^{\alpha}}$ if $n=0$. Since the right hand side is continuous for  $\kappa\geq -1 $, we conclude that the inequality holds in the limit $\kappa\to 1$. Using   
$1-R^{2} \leq  2(1-R)$ and solving for $R$ thus shows $R\leq 1-\big( A/(\pi4^{\alpha+1} B)\big)^{1/\alpha}$.

For $n\in\N$, we choose $\kappa=2$ and note that $1-R^3 \leq$ $ 3(1-R)$which allows us to conclude that $R\leq 1-\big(A/(4\pi 6^{2n+\alpha}B)\big)^{1/\alpha}$.\hfill $\Box$
    \end{proof}
\medskip
        \begin{remark}
        \textit{
      (1). One could improve the constants in Theorem~\ref{thm:gaps-wavelets} by optimizing the choice of $r$ in our proof. 
    However, due to page limitations, we chose not to pursue this idea.}

       \textit{ (2). If $\mu=\sum_{\lambda\in\Lambda}\delta_\lambda$ for some discrete set $\Lambda\subset \C^+$,  then $\mu$ is a sampling measure if and only if $\{\pi(\lambda)\psi_n^\alpha\}_{\lambda\in\Lambda}$ is a wavelet frame. Theorem~\ref{thm:gaps-wavelets} thus provides a bound on the maximal radius of a pseudohyperbolic disk that does intersect $\Lambda$.
        }

        \textit{(3). Our approach is independent of the Hilbert space structure and a slight adaptation would therefore also provide bounds on the gaps of $L^p$-sampling measures where the $L^p_\mu$-norm of $W_{\psi_n^\alpha} f$ is compared to the $L^p$-co-orbit space norm of $f$. We refer to \cite{berge_coorbit} for an introduction to co-orbit theory.} 
    \end{remark}

    \section{Bergman Spaces}\label{sec:bergman}
\medskip
 
Let $\mathcal{A}_\alpha(\C^+)$ be the space of holomorphic functions on the upper half plane that satisfy 
$$\|F\|^2_{\mathcal{A}_\alpha}:=\int_{\C^+}|F(z)|^2 s^{\alpha}d z<\infty.$$ 
    A  measure $\mu$ is called a \emph{sampling measure}  for  $\mathcal{A}_\alpha(\C^+)$ if 
$$
A\|F\|^2_{\mathcal{A}_\alpha}\leq \int_{\C^+}|F(z)|^2s^{\alpha} d\mu(z)\leq B\|F\|^2_{\mathcal{A}_\alpha}
$$
The \emph{Bergman transform}  $\text{B}_\alpha :\H^2(\C^+)\to\mathcal{A}_\alpha(\C^+)$
$$
\text{B}_\alpha f(z)=s^{-\frac{\alpha}{2}-1}W_{\psi_0^{\alpha+1}}f(z)
$$
is a constant multiple of an isometric isomorphism. This property immediately yields the following corollary of Theorem~\ref{thm:gaps-wavelets}.

\medskip
    \begin{corollary}
    \textit{
        Let $\mu$ be a sampling measure for   $\mathcal{A}_\alpha(\C^+)$ with   constants $A,B>0$. If $\mu( \mathcal{D}_R(z))= 0$ for some $z\in\C^+$, then     \begin{equation}\label{eq:cor-1}
    R\leq 1- \left(\frac{1}{  4^{\alpha+1}\pi}\ \frac{A}{ B}\right)^{1/\alpha},
    \end{equation}}
    \end{corollary}
\medskip
To the best of our knowledge, this result is not known even for the special case that  $\mu=\sum_{\lambda\in\Lambda}\delta_\lambda$ is a discrete measure , i.e., if $\Lambda$ is a \emph{set of stable sampling} for $\mathcal{A}_\alpha(\C^+)$.

\section{Sampling Measures for the short-time Fourier transform}
\medskip 
 The \emph{short-time Fourier transform  (STFT)} of a function $f\in L^2(\R)$ using a window $g\in L^2(\R)$ is given by
$$
V_gf(z)=\langle f,\pi(z)g\rangle, \quad z\in \C,
$$
where $\pi(z)g(t)=e^{2\pi i \xi t}g(t-x),$ $z=x+i\xi$. The STFT has, among other useful features, the covariance property
\begin{equation}\label{eq:cov}
V_g(\pi(w)f)(z)=e^{-2\pi i x'(\xi-\xi')} V_gf(z-w),
\end{equation}
$w=x'+i\xi'$, and satisfies \emph{Moyal's formula}
$$
\int_\C |V_g f(z)|^2dz=\|f\|^2_2\|g\|_2^2,\quad f,g\in L^2(\R).
$$
For a thorough introduction to time-frequency analysis we refer to \cite{groe1}.
A  measure $\mu$ is called a \emph{sampling measure} for the STFT if  there exist $A,B>0$ s.t. for any $f\in L^2(\R)$
$$
A\|f\|^2_{2} \leq \int_{\C}| V_g f(z)|^2  d\mu(z)\leq B\|f\|^2_{2}.
$$
If $\mu=\sum_{\lambda\in\Lambda}\delta_\lambda$ (for $\Lambda\subset \C$  discrete) is a sampling measure, then $\{\pi(\lambda)g\}_{\lambda\in\Lambda}$ forms a \textit{Gabor frame}, i.e.,
for $f\in L^2(\R)$
$$
A\|f\|^2_{2} \leq \sum_{\lambda\in\Lambda}| \langle f, \pi(\lambda)g\rangle |^2 \leq B\|f\|^2_{2}.
$$
 If $\{\pi(\lambda)g\}_{\lambda\in\Lambda}$ is a Gabor frame and if $|V_gg(z)|^2\lesssim (1+|z|)^{-(2+\sigma)}$, for some $\sigma>0$,  then Theorem~1.3 in  Papageorgiou and van Velthoven's paper \cite{jordy24} established quantitative bounds for the maximal radius of balls that do not intersect $\Lambda$. Their proof can be readily adapted for windows with exponential or Gaussian decay, as well as for  general sampling measures which would implicate a version of Theorem~\ref{thm:gaps-gabor} below. We nevertheless chose to include our  proof of Theorem~\ref{thm:gaps-gabor} based on the elementary ideas developed in Section~\ref{sec:wavelets} as it directly provides explicit constants.

We assume that the window function is chosen from the family of 
 Hermite functions $h_n(t):=c_n\hspace{1pt} e^{-\pi t^2}\frac{d}{dt}e^{-2\pi t^2}$, where $c_n$ is defined s.t.  $\|h_n\|_2=1$. It is well-known that  
$|V_{h_n}h_0(z)|=C_n|z|^ne^{-\pi|z|^2/2}$, for some $C_n\in\R^+$, see, e.g., \cite{abspe18-sieve}. Let us define the   auxiliary function
\begin{align}
H_n(z,w)&=\left|\frac{\langle \pi(0)h_0,\pi(z)h_n\rangle}{\langle \pi(w)h_0,\pi(z)h_n\rangle}\right|^2 \notag \\ &=\frac{|z|^{2n}}{|z-w|^{2n}}e^{-\pi (|z|^2-|z-w|^2)},\label{eq:H_n}
\end{align}
where we used \eqref{eq:cov}. Note that, just like in the wavelet case, $H_n$ is  rotationally invariant, i.e., $H_n(ze^{i\varphi},we^{i\varphi})=H_n(z,w)$, $\varphi\in(0,2\pi]$.
 \medskip
\begin{lemma}\label{lem:aux-stft}
    \textit{Let $n\in\N_0$, $w=R/2$, and $z=re^{i\varphi}$ with $r\geq R$ and $|\varphi|\leq \pi/5$. Then }
$$
H_n(z,w)\leq 4^n e^{-\pi R^2/2}.
$$    
\end{lemma}
\begin{proof}
First, we note that  
\begin{align*}|z-w|^{2}&=R^2/4+r^2-Rr \cos\varphi
\\ &\geq R^2/4+r^2-Rr=(r-R/2)^2\geq r^2/4.
\end{align*}
Second, if $|\varphi|\leq \pi/5$,
then $\cos\varphi\geq 3/4$. Consequently,
\begin{align*}
|z|^2-|z-w|^2&=r^2-R^2/4-r^2+Rr\cos\varphi \\ & \geq -R^2/4+3rR/4 \geq  R^2/2.
\end{align*}
Plugging the previous two estimates into \eqref{eq:H_n} yields 
$$
H_n(z,w)\leq \frac{r^{2n}}{r^{2n}/4^n}e^{-\pi R^2/2}=4^n e^{-\pi R^2/2},
 $$
which was to be shown. \hfill $\Box$
\end{proof}
 \medskip

\begin{theorem}\label{thm:gaps-gabor}
\textit{Let $n\in\N_0$, and $\mu$ be a sampling measure for the STFT with window $h_n$ and sampling bounds $A,B>0$.
If $\mu(\mathbb{D}_R(z))= 0$ for some $z\in\C$, then }
    \begin{equation}\label{eq:th-3}
    R^2\leq  \frac{2}{\pi}\log\left(   { 4^n5 }\  \frac{B}{A}\right).
    \end{equation}
\end{theorem}
\noindent\begin{proof}
     We  proceed similarly to  the proof of Theorem~\ref{thm:gaps-wavelets}.  
   First, we divide $\C\backslash\{0\}$ into $5$ segments   $$S_k=\big\{re^{i\varphi}:\ r>0,\ \pi(2k-1)/5\leq \varphi\leq \pi (2k+1)/5\big\},$$  and choose the points  $u_k=R/2\,  e^{2\pi i k/5},$ $k=0,1,..,5$. If $\mu(\mathbb{D}_R(z))=0$, then an  application of \eqref{eq:cov}, the rotational invariance of $H_n$  and Lemma~\ref{lem:aux-stft}   shows
    \begin{align*}
        A&=A\|\pi(z)h_0\|_2^2
        \\ 
        &\leq \hspace{-1pt}
        \int_{\C}\big|\langle \pi(z)h_0,\pi(w)h_n\rangle |^2d\mu(w)
       \\
         &  =  \hspace{-1pt}      \int_{\C }\big|\langle h_0,\pi( w\hspace{-1pt} -\hspace{-1pt} z)h_n\rangle |^2 d\mu(w)
          \\
         &  = \hspace{-1pt}  \sum_{k=1}^5     \int_{z+S_k}\hspace{-4pt} H_n(w\hspace{-1pt} -\hspace{-1pt} z,u_k) \big|\hspace{-.5pt}\langle \pi( u_k )h_0,\pi( w\hspace{-1pt} -\hspace{-1pt} z)h_n\rangle\hspace{-.5pt} \big|^2 d\mu(w)
       \\
          &\leq \hspace{-1pt}   \sum_{k=1}^5    \hspace{-1pt}\sup_{\gamma\in S_k\backslash \mathbb{D}_R}\hspace{-4pt} H_n(\gamma,u_k) \hspace{-2pt}\int_{\C}\hspace{-1pt}\big|\hspace{-.5pt}\langle \pi( z\hspace{-1pt} +\hspace{-1pt} u_k)h_0,\pi(w)h_n\rangle \hspace{-.5pt}\big|^2d\mu(w)
      \\
        &\leq  5  B  \sup_{\gamma\in S_0\backslash \mathbb{D}_R }\hspace{-2pt}H_n(\gamma,u_0) \leq  4^n5Be^{-\pi R^2/2}.
    \end{align*}  
    Solving for $R$ then completes the proof.\hfill $\Box$
\end{proof}
\medskip
\begin{remark}
\textit{Like in Section~\ref{sec:bergman}, the bounds of Theorem~\ref{thm:gaps-gabor}  can be directly translated to bounds for sampling measures on the Bargmann-Fock space of entire functions ($n=0$), and to (true) polyanalytic Bargmann-Fock spaces ($n\geq 1$), see \cite{abgroe12}.  }
\end{remark}

\section*{Acknowledgements}

\noindent 
This research was funded by the Austrian Science Fund (FWF)  10.55776/PAT1384824.

    \bibliographystyle{plain}
\bibliography{paperbib}
\end{document}